\documentclass[a4paper,12pt]{article}
\usepackage{amsfonts,amssymb,amsmath}
\usepackage[english]{babel}
\usepackage{bm,theorem}

\title{Isometries and Collineations\\ of the Cayley Surface}
\author{Johannes Gmainer \and Hans Havlicek}

\date{}

\theoremstyle{change} {\theorembodyfont{\rm}
  \newtheorem{defi}{Definition}[section]

}
  \newtheorem{lemma}[defi]{Lemma}
  \newtheorem{proposition}[defi]{Proposition}
  \newtheorem{theorem}[defi]{Theorem}

\newcommand{\Matrixfeld}[4]{\left#1\!\begin{array}{*{#3}{c}}#4\end{array}\!\right#2}

\newcommand{\Mat}{\Matrixfeld()}

\newenvironment{proof}
    {\begin{trivlist} \item {\emph{Proof}.}} 
    {{}\hfill $\square$ \end{trivlist}} 

\newcommand{\Aut}{{\mathrm{Aut}}}
\DeclareMathOperator{\Char}{{\mathrm{Char}}}
\newcommand{\DV}{{\mathrm{CR}}}

\newcommand{\GL}{{\mathrm{GL}}}
\newcommand{\id}{{\mathrm{id}}}
\newcommand{\Stab}{{\mathrm{Stab}}}

\newcommand{\T}{{\mathrm{T}}}
\newcommand{\kreis}{{\mathcal{C}}}

\newcommand{\cE}{{\mathcal E}}

\newcommand{\cR}{{\mathcal R}}
\newcommand{\cV}{{\mathcal V}}

\newcommand{\PP}{{\mathbb P}}

\newcommand{\RR}{{\mathbb R}}

\newcommand{\bbp}{{\bm p}}

\newcommand{\bX}{{\bm X}}

\newcommand{\eps}{{\varepsilon}}

\newcommand{\notpar}{\mathbin{\not{}\hspace{-0.35ex}{\parallel}\hspace{0.35ex}}}
\begin{document}
\sloppy

\maketitle

\begin{abstract}
Let $F$ be Cayley's ruled cubic surface in a projective three-space over any
commutative field $K$. We determine all collineations fixing $F$, as a set, and
all cubic forms defining $F$. For both problems the cases $|K|=2,3$ turn out to
be exceptional. On the other hand, if $|K|\geq 4$ then the set of simple points
of $F$ can be endowed with a non-symmetric distance function. We describe the
corresponding circles, and we establish that each isometry extends to a unique
projective collineation of the ambient space.
\end{abstract}

\begin{section}{Introduction}

We investigate Cayley's ruled cubic surface $F$ in a three-dimensional
projective space over an arbitrary commutative ground field $K$. It is fairly
obvious that ``most'' of the results that are known from the classical case
($K=\mathbb{R},\mathbb{C}$) will remain valid. However, a closer look shows
that the situation is sometimes rather intricate.

\par

In Section \ref{se:koll} we determine all collineations of the Cayley surface.
If $|K|\leq 3$ then there are ``more'' such collineations than in the general
case. From the proof of this result it is immediate that for $|K|\leq 3$ there
are non-proportional cubic forms defining $F$. However, that proof does not
answer the question of finding all such cubic forms to within a non-zero
factor. We pay attention to this question, since it governs the interplay
between incidence geometry and algebraic geometry. In Section \ref{se:alle} we
show that the number of solutions to this problem equals $64$ if $|K|=2$, two
if $|K|=3$, and one otherwise. Our first attempt was to solve this problem by
``brute force'' with the help of a computer algebra system. However, due to the
presence of polynomial identities of high degree, we could not succeed without
assuming $|K|$ being rather large. Therefore, in our current approach, we first
use a lot of geometric reasoning before we enter the realm of algebra. In this
way we obtain the result for $|K|\geq 3$. By virtue of a theorem due to
\textsc{G.~Tallini} \cite{tallini-61a}, it is easy to treat the remaining case
$|K|=2$.

\par

The Cayley surface has an interesting ``inner geometry'' which can be based
upon a distance function appearing (in the real case) in an article of
\textsc{H.~Brauner} \cite{brau-64}. In Section~\ref{se:distance-function},
using a completely different, purely geometrical approach, we generalize this
distance function to the case of an arbitrary ground field $K$ with more than
three elements. Our distance function $\delta$ fits into the very general
concepts developed by \textsc{W.~Benz} \cite{benz-92}. It is non-symmetric;
this means that the distance from $A$ to $B$ is in general not the distance
from $B$ to $A$. It will be established that $\delta$ is a defining function
for the group of automorphic projective collineations of the Cayley surface.

\par

Occasionally, we shall also come across phenomena reflecting the characteristic
of the ground field, like the presence of a line of nuclei in case $\Char K=3$
(cf.\ formula~(\ref{partial-null-zwei-eq})), or the absence of circles with
more than one mid-point in case of $\Char K=2$ (cf.\
Proposition~\ref{prop:kreise}).
\end{section}

\begin{section}{Preliminaries}\label{se:prelim}

Throughout this article we consider the three-dimensional projective space
$\PP_3(K)$ over a commutative field $K$. The points of $\PP_3(K)$ are the
one-dimensional subspaces of the column space $K^{4 \times 1}$, viz.\ they are
of the form $K \bbp$ with $(0,0,0,0)^{\T} \ne \bbp = (p_0,p_1,p_2,p_3)^{\T} \in
K^{4 \times 1}$.

\par

Let $K[X_0,X_1,X_2,X_3]$ be the polynomial ring which arises from $K$ by
adjoining independent indeterminates $X_0, X_1, X_2, X_3$. We shall use the
shorthand $\bX:=(X_0,X_1,X_2,X_3)$. Each polynomial $g(\bX) \in K[\bX]$
determines a \emph{polynomial function\/}
\begin{equation}\label{eq:polyfkt}
K^{4 \times 1} \rightarrow K: (p_0, p_1, p_2, p_3)^{\T} \mapsto g(p_0, p_1,
p_2, p_3).
\end{equation}
Since $K$ may be a finite field, it is necessary to distinguish between a
polynomial and the associated polynomial function. We shall mainly be concerned
with homogeneous polynomials (forms) in $K[\bX]$. By virtue of
(\ref{eq:polyfkt}), the subspace of homogeneous polynomials of degree one in
$K[\bX]$ is in bijective correspondence with the space of linear mappings $K^{4
\times 1} \rightarrow K$ (the dual space of $K^{4 \times 1}$), which in turn
can be viewed as the row space $K^{1 \times 4}$. This bijection allows to
identify $K[\bX]$ with the \emph{symmetric algebra\/} on the row space $K^{1
\times 4}$; cf., for example, \cite[pp.~155--156]{deheu-93}.

\par

We refer to \cite[pp.~48--51]{hirs-98} for those basic notions of algebraic
geometry which will be used in this paper. However, our notation differs from
\cite{hirs-98}, as we write, for example,
\begin{equation*}\label{}
 \cV\big(g_1(\bX),\ldots,g_r(\bX)\big)
 :=\big\{K\bbp\in\PP_3(K)\mid g_1(\bbp)=\cdots=g_r(\bbp)=0\big\}
\end{equation*}
for the set of $K$-rational points of the variety given by homogeneous
polynomials $g_1(\bX),g_2(\bX),\ldots,g_r(\bX) \in K[\bX]$.

\par

The plane $\omega:=\cV(X_0)$ will be considered as \emph{plane at infinity},
thus turning $\PP_3(K)$ into a projectively closed affine space. Finally, let
$Q_i:= K (\delta_{0i},\delta_{1i},\delta_{2i},\delta_{3i})^{\T}$, where
$\delta_{ji}$ is the Kronecker delta and $i\in\{0,1,2,3\}$, be the base points
of the standard frame of reference.

\par

Let us turn to \emph{Cayley's ruled cubic surface\/} or, for short, the
\emph{Cayley surface}. It is, to within projective collineations, the point set
$F:=\cV\big(f(\bX)\big)$, where
\begin{equation}\label{eq:aff-punkte}
 f(\bX):=X_0X_1X_2-X_1^3-X_0^2X_3 \in
 K[\bX].
\end{equation}
We collect some of its properties (see \cite{bertini-23}, \cite{brau-64}, and
\cite{sauer-37} for the classical case): The parametrization
\begin{equation}\label{eq:P(u,v)}
 K^2\to \PP_3(K): (u_1,u_2) \mapsto K(1,u_1,u_2,u_1u_2-u_1^3)^{\T}=:P(u_1,u_2)
\end{equation}
is injective, and its image coincides with $F\setminus\omega$ (the affine part
of $F$). The intersection
\begin{equation}\label{g-unendlich-eq}
F\cap\omega=\cV(X_0,X_1)=:g_{\infty}
\end{equation}
is a line. By the above, the Cayley surface has $|K|^2+|K|+1$ points; cf.\
\cite[Teorema~6]{rosati-56}. Hence, in case of a finite ground field, it does
not fit into the characterizations given by \textsc{G.~Tallini}
\cite{tallini-59}. The plane $\cV(X_3)$ meets $F$ along the line $\cV(X_1,X_3)$
and the parabola
\begin{equation}\label{kegelschnitt-eq}
    l:=\cV(X_0X_2-X_1^2,X_3).
\end{equation}
The mapping
\begin{equation}\label{erzeugung-eq}
    \beta: l \rightarrow g_{\infty}: K(s_0^2,s_0s_1,s_1^2,0)^{\T}
    \mapsto K(0,0,s_0,s_1)^{\T},
\end{equation}
where $(0,0)\neq(s_0,s_1)\in K^2$, is projective, and each point of $l$ is
distinct from its image point. Let $g(s_0,s_1)$ denote the line joining the two
points given in (\ref{erzeugung-eq}). Thus, in particular, we obtain
$g(0,1)=g_{\infty}$.

\par

It is immediate that every line $g(s_0,s_1)$ is a \emph{generator\/} of $F$,
i.e., it is contained in $F$. Conversely, let $h \subset F$ be a line. If $h
\subset \omega$ then $h=g_\infty$, by (\ref{g-unendlich-eq}). Otherwise, $h$
has a unique point at infinity which necessarily belongs to $g_{\infty}$. The
plane which is spanned by $h$ and $g_{\infty}$ has the form
$\pi=\cV(a_0X_0-X_1)$ with $a_0 \in K$. The intersection $F \cap \pi$ is given
by
\begin{equation}\label{eq:2geradenschnitt}
\cV\left(a_0X_0-X_1, X_0^2(a_0X_2-a_0^3X_0-X_3)\right),
\end{equation}
whence it consists of two distinct lines. This shows $F \cap \pi = g_{\infty}
\cup h$ and $h=g(1,a_0)$.

\par

According to (\ref{erzeugung-eq}), the line $g_{\infty}$ is not only a
generator of $F$, but also a \emph{directrix}, as it has non-empty intersection
with every generator. Each point of $g_{\infty}$, except the point $Q_3$, is on
precisely two generators of $F$; each affine point of $F$ is incident with
precisely one generator. Thus the projectivity (\ref{erzeugung-eq}) can be used
to ``generate'' the Cayley surface in a purely geometric way. This is nicely
illustrated in \cite[p.~89]{sauer-37} for the real projective three-space.
\end{section}

\begin{section}{Automorphic collineations of $F$}\label{se:koll}

Each matrix $M=(m_{ij})_{0\leq i,j\leq 3}\in\GL_4(K)$ acts on the row space
$K^{1 \times 4}$ by multiplication from the right hand side. By identifying
each row vector $(d_0, d_1, d_2, d_3)\in K^{1 \times 4}$ with $d_0X_0+
d_1X_1+d_2X_2+d_3X_3\in K[\bX]$, the matrix $M$ yields a linear bijection of
the subspace of homogeneous polynomials of degree one; in particular,
\begin{equation}\label{eq:poly_aktion}
    X_i \mapsto \sum_{j=0}^{3} m_{ij}X_j \mbox{ for }i\in\{0,1,2,3\}.
\end{equation}
By the universal property of symmetric algebras, this linear bijection extends
to a $K$-algebra automorphism of $K[\bX]$; cf., e.g., \cite[p.~156]{deheu-93}.
Thus, altogether, $\GL_4(K)$ acts on $K[\bX]$.

\par

On the other hand, $M$ acts on the column space $K^{4 \times 1}$ by left
multiplication, and therefore as a projective collineation on $\PP_3(K)$. Given
a form $g(\bX)\in K[\bX]$ and its image under $M$, say $h(\bX)$, this
collineation takes $\cV\big(h(\bX)\big)$ to $\cV\big(g(\bX)\big)$, since
$g(M\cdot\bbp)=h(\bbp)$ for all $\bbp\in K^{4\times 1}$. If, moreover,
$h(\bX)\sim g(\bX)$, i.e., the polynomials are proportional by a non-zero
scalar in $K$, then $\cV\big(g(\bX)\big)= \cV\big(h(\bX)\big)$.

\par

The following result holds:

\begin{lemma} The set of all matrices
\begin{equation}\label{zwei-koll-typen-eq}
  M_{a,b,c}:= \Mat4{
  1&0&0&0\\
  a&c&0&0\\
  b&3\,ac&{c}^{2}&0\\
  ab-{a}^{3}&bc&a{c}^{2}&{c}^{3}
  }
\end{equation}
where $a,b \in K$ and $c \in K^{\times}:=K\setminus\{0\}$ is a group, say
$G(F)$, under multiplication. Each matrix in $G(F)$ leaves invariant the cubic
form $f(\bX)= X_0X_1X_2-X_1^3-X^{2}_0X_3$ to within the factor $c^3$.
Consequently, the group $G(F)$ acts on $F$ as a group of projective
collineations.
\end{lemma}
\begin{proof}
We obtain, for all $a,b,c,x,y,z\in K$ with $c,z\neq 0$,
$M_{a,b,c}^{-1}=M_{a',b',c'}$, where $a':=-ac^{-1}$, $b':=(3a^2-b)c^{-2}$,
$c':=c^{-1}$, and $M_{a,b,c}\cdot
 M_{x,y,z}=M_{a'',b'',c''}$, where $a'':=a+cx$, $b'':=b+3acx+c^2y$, $c'':=cz$.
The rest is a straightforward calculation: By (\ref{eq:poly_aktion}), the image
of $f(\bX)$ under the action of $M_{a,b,c}$ equals $c^3f(\bX)$.
\end{proof}
\begin{lemma} \label{lem:semilin}
For each automorphism $\zeta \in \Aut(K)$ the collineation $\PP_3(K)\to
\PP_3(K) : K (p_0,p_1,p_2,p_3)^\T
 \mapsto K\big(\zeta(p_0),\zeta(p_1),\zeta(p_2),\zeta(p_3)\big)^\T$
leaves invariant the Cayley surface $F$.
\end{lemma}
\begin{proof}
Observe that all coefficients of the polynomial $f(\bX)$ are in the prime field
of $K$, whence they are fixed under $\zeta$. Therefore
$f(\bbp)=f\big(\zeta(\bbp)\big)$ for all $\bbp\in K^{4\times 1}$.
\end{proof}
We now turn to the problem of finding all automorphic collineations of $F$. The
following lemma is preliminary, a stronger result will be established in
Theorem \ref{thm:trans}.

\begin{lemma} \label{transitiviteats-thm}
The group $G(F)$ acts transitively on $F \setminus g_{\infty}$.
\end{lemma}
\begin{proof}
We fix the base point $Q_0\in F\setminus g_\infty$. By (\ref{eq:aff-punkte}),
an arbitrarily chosen affine point of $F$ has the form $P(u_1,u_2)$ with
$(u_1,u_2)\in K^2$. Hence the matrix $M_{u_1,u_2,1}$ takes $Q_0=P(0,0)$ to
$P(u_1,u_2)$, and the assertion follows.
\end{proof}
We remark that $\{M_{a,b,1}\mid a,b\in K\}$ is a commutative subgroup of
$G(F)$. By the previous proof, this group acts regularly on $F\setminus
g_\infty$. Summing up our three lemmas, we obtain
\begin{proposition}
Each collineation $\kappa$ of\/ $\PP_3(K)$ which fixes the Cayley surface $F$
can be written as $\kappa=\kappa_3\circ\kappa_2\circ\kappa_1$, where $\kappa_1$
is given as in Lemma \emph{\ref{lem:semilin}}, $\kappa_2$ is a projective
collineation which stabilizes $F$ and the base point $Q_0=K(1,0,0,0)$, and
$\kappa_3$ is induced by a matrix in $G(F)$.
\end{proposition}
We are thus lead to our first main result:

\begin{theorem}\label{thm:stabil}
Let $\Stab(F,Q_0)$ be the group of all projective collineations of $\PP_3(K)$
which stabilize $F$ and the base point $Q_0$. Depending on the ground field
$K$, this stabilizer is determined by the following subgroups of $\GL_4(K)$.
\begin{eqnarray}
    |K|=2       &:& \{ M_{0,0,1}, N \},\\
    |K|=3       &:& \{ M_{0,0,c}, N_c \mid c\in K^\times \},\\
    \label{eq:stabil4}
    |K|\geq 4   &:& \{ M_{0,0,c}\mid c\in K^\times \},
\end{eqnarray}
where
\begin{equation}\label{M-4-eq}
    N:= \Mat4{
      1 & 0 & 0 & 0 \\
      0 & 1 & 0 & 0 \\
      0 & 1 & 1 & 0 \\
      0 & 1 & 0 & 1   }
\mbox{ and }
    N_{c}:= \Mat4{
      1 & 0 & 0 & 0 \\
      0 & c & 0 & 0 \\
      0 & 0 & 2 & 0 \\
      0 & c & 0 & 2c    }.
\end{equation}
\end{theorem}
\begin{proof}
Let $\sigma\in\Stab(F,Q_0)$. We saw at the end of Section~\ref{se:prelim} that
only the points of $g_\infty\setminus\{Q_3\}$ are on two distinct generators,
whereas each other point of $F$ is incident with one generator only. Thus
$\sigma(g_{\infty})=g_{\infty}$, and $\sigma(Q_3)=Q_3$. Also, $\omega$ is the
only plane through $g_\infty$ which does not contain a second generator, so
that $\sigma(\omega)=\omega$. Also, since $Q_0$ is fixed, so is the only
generator $g(1,0)$ through this point, whence $g(1,0)\cap g_\infty$, i.e.\ the
base point $Q_2$, is fixed too. Consequently, $\sigma$ is induced by a lower
triangular matrix
\begin{equation*}
(x_{ij})=\Mat4{
  1 & 0 & 0 & 0 \\
  0 & x_{11} & 0 & 0 \\
  0 & x_{21} & x_{22} & 0 \\
  0 & x_{31} & 0 & x_{33}  } \in\GL_4(K).
\end{equation*}
It remains to determine the unknown entries of this matrix, where obviously
\begin{equation}\label{eq:det}
 \det(x_{ij})= x_{11}x_{22}x_{33}\neq 0.
\end{equation}
First, fix a scalar $t\in K^\times$ and consider the generator $g(1,t)$. There
is an $s\in K^\times$ such that $\sigma\big(g(1,t)\big)=g(1,s)$. Thus for each
$\lambda\in K$ exists an element $\mu\in K$ with $\sigma(K (1, t,
t^2+\lambda,\lambda t)^{\T})=K (1, s, s^2+\mu,\mu s)^{\T}$. So
\begin{eqnarray}
  x_{11}t &=& s, \label{eq:1}\\
  x_{21}t+x_{22}\left(t^2+\lambda\right) &=& s^2+\mu, \label{eq:2}\\
  x_{31}t+x_{33} \lambda t &=&\mu s .\label{eq:3}
\end{eqnarray}
We divide (\ref{eq:3}) by $s$, subtract it then from (\ref{eq:2}), and
substitute $s=x_{11}t$ according to (\ref{eq:1}). Hence
\begin{equation}\label{eq:4}
  x_{21}t + x_{22}t^2 - \frac{x_{31}}{x_{11}}+
                              \left(x_{22}  - \frac{x_{33}}{x_{11}}\right)\lambda
  =
  x_{11}^2t^2 \mbox{ for all }\lambda\in K.
\end{equation}
This implies
\begin{equation}\label{eq:x_33}
  x_{33} = x_{11}x_{22}.
\end{equation}
Next, we assume $t$ to be variable, whence (\ref{eq:4}) gives
\begin{equation}\label{quad-poly-in-s}
    \left({x_{22}} - {x_{11}^2}\right)t^2 + {x_{21}}t-\frac{x_{31}}{x_{11}}=0
 \mbox{ for all } t\in K^{\times}.
\end{equation}
According to the cardinality of $K$ there are three cases:

\par

$|K|=2$: By (\ref{eq:det}), $x_{11}=x_{22}=x_{33}=1$ and (\ref{quad-poly-in-s})
reads $x_{21}\cdot 1=x_{31}$, so that $(x_{ij})=M_{0,0,1}$ or $(x_{ij})=N$.

\par

$|K|=3$: Then $x_{11}^2=1=t^2$ for $t \in \{1,2\}= K^{\times}$, and there are two
possibilities: (i) $x_{22}=1$, whence (\ref{quad-poly-in-s}) reads
$x_{21}t-{x_{31}}/{x_{11}}=0$ for $t \in \{1,2\}$, so that $x_{21}=x_{31}=0$,
and $(x_{ij})=M_{0,0,c}$ with $c:=x_{11}$. (ii) $x_{22}=2$, whence
(\ref{quad-poly-in-s}) turns into $1 \cdot 1 +x_{21}t-{x_{31}}/{x_{11}}=0$ for
$t \in \{1,2\}$, so that $x_{21}=0$, $x_{31}=x_{11}$, and $(x_{ij})=N_{c}$ with
$c:=x_{11}$.

\par

$|K| \geq 4$: From $|K^{\times}| \geq 3$ and (\ref{quad-poly-in-s}) follows
$x_{22}=x_{11}^2$, $x_{21}=x_{31}=0$, and $(x_{ij})=M_{0,0,c}$ with
$c:=x_{11}$.

\par

In either case it is easy to see that the given matrices form a subgroup of
$\GL_4(K)$.
\end{proof}
We denote by $G_\mathrm{ext}(F)$ the \emph{extended group\/} of the Cayley
surface, i.e. the group of \emph{all\/} matrices $(x_{ij})_{0\leq i,j,\leq
3}\in\GL_4(K)$, subject to the condition $x_{00}=1$, leaving invariant the
Cayley surface $F$. By the first paragraph of the previous proof, each
automorphic projective collineation of $F$ is induced by precisely one matrix
in $G_\mathrm{ext}(F)$. Furthermore, for $|K| \geq 4$, we have
$G(F)=G_\textrm{ext}(F)$, whereas for $|K| \leq 3$ the groups $G(F)$ and
$G_\textrm{ext}(F)$ are distinct, since none of the matrices $N$ and $N_c$ is
in $G(F)$. We are thus lead to the following result:

\begin{proposition}
Let
\begin{equation}\label{eq-andere_poly_fuer_F}
    f_{(|K|)}(\bX):=\left\{\!\!
\renewcommand{\arraystretch}{1.}
\begin{array}{l}
    X_0X_1^2+X_0X_1X_2+X_1^3+ X_0^2X_1+X_0^2X_3\mbox{ when }|K|=2, \\
    2X_0X_1X_2+2X_1^3+2X_0^2X_1+X_0^2X_3\mbox{ when }|K|=3.
\end{array}
\right.\!
\end{equation}
Then, for $|K|\leq 3$, the Cayley surface $F=\cV\big(f(\bX)\big)$ coincides
with $\cV\big(f_{(|K|)}(\bX)\big)$.
\end{proposition}

\begin{proof}
Let $|K|=2$. The image of $f(\bX)$ under the action of $N$ gives the polynomial
$f_{(2)}(\bX)$. Likewise, for $|K|=3$, the polynomial $f_{(3)}(\bX)$,
multiplied by $c\in\{1,2\}$, arises as the image of $f(\bX)$ under the action
of $N_c$.
\end{proof}
Observe that here ``to coincide'' just refers to sets of points and not to
algebraic varieties in the sense of \cite[p.~48]{hirs-98}. Thus, for $|K|\leq
3$, the point set of the Cayley surface $F$ may also be considered as the
algebraic curve $\cV\big(f(\bX),f_{|K|}(\bX)\big)$.
\end{section}

\begin{section}{All cubic forms defining $F$}\label{se:alle}

In discussing the Cayley surface $F$ we have to distinguish between properties
which stem from the defining polynomial $f(\bX)$ and geometric properties,
i.e., properties which are invariant with respect to the action of the group
$G_\textrm{ext}(F)$. First, we recall some notions which can be defined in
terms of $f(\bX)$. Let $\partial_i:=\frac{\partial}{\partial X_i}$. We start by
calculating the partial derivatives
\begin{equation}\label{partial-eq}
\begin{array}{rcl@{\quad}rcl}
    \partial_0 f(\bX) &=& X_1X_2-2X_0X_3,&
    \partial_1 f(\bX) &=& X_0X_2-3X_1^2,\\
    \partial_2 f(\bX) &=& X_0X_1,&
    \partial_3 f(\bX) &=& -X_0^2.
\end{array}
\end{equation}
They vanish simultaneously at $(p_0,p_1,p_2,p_3)^{\T} \in K^{4 \times 1}$ if,
and only if, at least one of the following conditions holds:
\begin{eqnarray}
    \label{partial-null-eins-eq}
    & p_0=p_1=0; &\\
    \label{partial-null-zwei-eq}
    & p_0=p_2=0 \mbox{ and } \Char K=3. &
\end{eqnarray}
When $K$ is a field of characteristic $\Char K=3$ then, by
(\ref{partial-null-zwei-eq}), $\cV(X_0,X_2)$ is a distinguished line in the
ambient space of the Cayley surface $F$. Each of its points is a {\em nucleus}
of $F$. See \cite[p.~50]{hirs-98} and \cite[Proposition~3.17]{definis+d-83},
where nuclei are defined in a slightly different way. All points subject to
(\ref{partial-null-eins-eq}) are \emph{singular}; they comprise the line
$g_\infty\subset F$. We obtain, for all $s_2, s_3 \in K$, that
\begin{equation*}
  f\big((0,0,s_2,s_3)+T\bX\big)=
  T^2X_0(s_2X_1-s_3X_0)+ T^3(*)  \in K[\bX,T].
\end{equation*}
Hence all points of $g_{\infty}$ are \emph{double points}. The \emph{tangent
cone\/} (see \cite[p.~49]{hirs-98}, where the term \emph{tangent space\/} is
used instead) at a point $Y=K(0,0,s_2,s_3)^\T$, $(s_2,s_3)\neq (0,0)$, is
\begin{equation}\label{eq:tangential-unendlich}
  \cV\big(X_0(s_2X_1-s_3X_0)\big),
\end{equation}
whence we refer to the plane at infinity as the \emph{tangent plane\/} at
$Y=Q_3$. For $Y=K(0,0,1,s_3)^\T$ the tangent cone is the union of the plane at
infinity and the plane spanned by $g_{\infty}$ and the generator $g(1,s_3)$. We
call each of these planes a \emph {tangent plane\/} at $Y$. By
(\ref{partial-eq}), all points of $F\setminus g_\infty$ are \emph{simple}. The
\emph{tangent plane\/} at $P(u_1,u_2)$ (see (\ref{eq:P(u,v)})) equals
\begin{equation}\label{tangentialebene-eq}
   \cV\left( (2u_1^3-u_1u_2)X_0 + (-3u_1^2+u_2)X_1 + u_1X_2 - X_3\right).
\end{equation}

\par
Next, we present a characterization of tangent planes:

\begin{proposition}\label{prop:tg-ebenen}
Let $\tau=\cV\left(\sum_{i=0}^{3}a_iX_i\right)$, where $a_i\in K$, be a plane.
Then the following assertions are equivalent.
\begin{enumerate}
    \item $\tau$ is a tangent plane of $F$ with respect to $f(\bX)$.

    \item\label{eq:dualcf} The coefficients $a_i$ satisfy the equation
              $a_0a_3^2 - a_1a_2a_3 + a_2^3 =0$.

    \item $\tau$ contains a generator of $F$.
\end{enumerate}
\end{proposition}
\begin{proof}
(a) $\Rightarrow$ (b): If $\tau$ is the tangent plane at an affine point of $F$
then the coefficients $a_i$ are proportional to the coefficients of a
polynomial as in (\ref{tangentialebene-eq}), otherwise we obtain $a_2=a_3=0$.
In any case (\ref{eq:dualcf}) holds.

\par

(b) $\Rightarrow$ (c): If $a_3=0$ then so is $a_2$. Consequently
$g_\infty\subset \tau$. If $a_3\neq 0$ then we may let w.l.o.g. $a_3=-1$,
whence $g(1,a_2)\subset\tau$.

\par

(c) $\Rightarrow$ (a): This is immediate from (\ref{eq:tangential-unendlich})
and (\ref{tangentialebene-eq}).
\end{proof}
By (\ref{eq:dualcf}), the set of all tangent planes with respect to $f(\bX)$ is
a Cayley surface in the dual projective space. In view of
(\ref{tangentialebene-eq}), it is somewhat surprising that this holds
irrespective of the characteristic of $K$.

\par

Clearly, the notions from the above are not independent of the homogeneous
polynomial which is used for defining $F$. For example, we have
$F=\cV\big(f(\bX)^2\big)$, but no point of $F$ is simple with respect to
$f(\bX)^2$. However, by restricting ourselves to cubic forms defining $F$, we
obtain the next two theorems.

\begin{theorem}
Let $|K|\geq 3$ and suppose that $p(\bX)\in K[\bX]$ is a cubic form such that
$\cV\big(p(\bX)\big)$ equals the Cayley surface $F=\cV\big(f(\bX)\big)$. Then
$p(\bX)\sim f(\bX)$ or, only for $|K|=3$, $p(\bX)\sim f_3(\bX)$, where
$f_3(\bX)$ is given by \emph{(\ref{eq-andere_poly_fuer_F})}.
\end{theorem}
\begin{proof}
(a) Suppose that $p(\bX) = \sum_{0 \leq i \leq j \leq k \leq 3}
a_{ijk}X_iX_jX_k$ is a form of degree three such that $\cV\big(p(\bX)\big)=F$.
We aim at finding the twenty coefficients $a_{ijk}$ to within a common non-zero
factor, and we adopt the following convention: Within this proof, concepts like
``simple point'', ``double point'', ``intersection multiplicity'', and
``tangent plane'' are tacitly understood with respect to $p(\bX)$, unless
explicitly stated otherwise.

\par

Obviously, $Q_0,Q_2,Q_3\in F$, whereas $Q_1\notin F$. Hence
\begin{equation*}\label{}
    a_{111}\neq a_{000}=a_{222}=a_{333}=0.
\end{equation*}
The line $g(1,0)$ is on $F$, whence
$p\big((1,0,t,0)^\T\big)=t(a_{002}+a_{022}t)=0$ for all $t\in K$. From $|K|\geq
3$, we obtain
\begin{equation*}\label{}
    a_{002}=a_{022}=0.
\end{equation*}
Likewise, $g_\infty\subset F$ forces
\begin{equation*}\label{}
    a_{223}=a_{233}=0.
\end{equation*}

\par

(b) We proceed by establishing four auxiliary results:

\par

(I) \emph{Each affine point $Y\in F$ is simple\/}. It suffices to show that
there exists a line $m\ni Y$ such that $|m\cap F|= 3$, since such a line meets
$F$ at $Y$ with multiplicity one. First, let $Y=Q_0$. We consider the line
$m_0$, say, joining $Q_0$ and $K(0,1,\alpha+1,\alpha)^\T\notin F$, where
$\alpha\in K\setminus\{0,1\}$; such an $\alpha$ exists by $|K|\geq 3$. The
intersection $m_0\cap F$ equals the set of all points
$K\big(1,\xi,\xi(\alpha+1),\xi\alpha\big)^\T$, $\xi\in K$, with
\begin{equation*}\label{}
    f\big((1,\xi,\xi(\alpha+1),\xi\alpha)^\T\big)= \xi^2(\alpha+1)-\xi^3-\xi\alpha =
    -\xi(\xi-1)(\xi-\alpha)=0.
\end{equation*}
Hence $m_0$ has the required property. By the transitive action of $G(F)$ on
$F\setminus g_\infty$, the assertion follows for all $Y\in F\setminus
g_\infty$.

\par

(II) \emph{$\cV(X_3)$ is the tangent plane of at least one affine point, say
$R$, on $g(1,0)$.} We know from (\ref{kegelschnitt-eq}) that $F\cap\cV(X_3)$ is
the union of the generator $g(1,0)$ and the parabola $l$. Hence $p(\bX)=
q(\bX)X_1+X_3(*)$, where
\begin{equation*}\label{}
  q(\bX)=a_{001}X_0^2+a_{011}X_0X_1+a_{012}X_0X_2+a_{111}X_1^2+
     a_{112}X_1X_2+a_{122}X_2^2.
\end{equation*}
The planar quadric $\widetilde l:=\cV(q(\bX),X_3)$ and the parabola $l$ have
the same points outside the line $g(1,0)$. There are at least two such points
because of $|l|=|K|+1\geq 4$. Therefore $|\widetilde l|\geq 2$. It is well
known that a planar quadric with at least two points is either a
(non-degenerate) conic, a pair of distinct lines, or a repeated line. As
$g(1,0)$ is the only line contained in $F\cap\cV(X_3)$, we see that $\widetilde
l$ has to be a conic. If $|K|$ is finite then $|l|=|\widetilde l|$ implies that
$|g(1,0)\cap\widetilde l|=2$; thus we can choose a point $R\in
\big(g(1,0)\cap{\widetilde l}\,\big)\setminus\omega$. If $|K|$ is infinite then
$l$ and $\widetilde l$ have infinitely many common points outside $g(1,0)$. So
we obtain $l=\widetilde l$, and we let $R:=Q_0$. In any case, the tangent plane
of $F$ at $R$ contains the tangent of $\widetilde l$ at $R$ (which is also a
tangent of $F$ with respect to $p(\bX)$), and the generator $g(1,0)$. As these
two lines do not coincide, the tangent plane of $F$ at $R$ is $\cV(X_3)$.

\par

(III) \emph{The tangent plane at each affine point of $F$ does not pass through
$Q_3$.} The planar section $F\cap\cV(X_1)$ consists precisely of the two lines
$g(1,0)$ and $g_\infty$. Therefore
\begin{equation*}
 p(\bX)= X_0\underbrace{(a_{003}X_0+a_{023}X_2+a_{033}X_3)}
                       _{=:\,r(\bX)}X_3+X_1(*),
\end{equation*}
where $r(\bX)\sim X_0$ or $r(\bX)\sim X_3$, whence
\begin{equation}\label{eq:geradenpaar1}
    a_{023}=0.
\end{equation}
Moreover, precisely one of the coefficients $a_{003}$ and $a_{033}$ vanishes.
We claim that
\begin{equation}\label{eq:geradenpaar2}
    a_{003}\neq a_{033}=0.
\end{equation}
Assume to the contrary that $a_{003}=0\neq a_{033}$. Hence we would have
$p(\bX)=a_{033}X_0X_3^2+X_1(*)$. Then the line joining $Q_3$ with the point
$R\in g(1,0)$ from (II) would meet $F$ at $R$ with multiplicity two, whence the
tangent plane at the simple point $R$ would be $\cV(X_1)$, a contradiction to
(II).

\par

Next, choose any affine point $Y\in g(1,0)$. The line $Q_3Y$ meets $F$ at $Y$
with multiplicity one, due to (\ref{eq:geradenpaar1}) and
(\ref{eq:geradenpaar2}). Thus it is not a
 tangent, and the assertion follows for all affine points of $g(1,0)$.

\par

Finally, consider an arbitrary affine point $Y$ of $F$. By Lemma
\ref{transitiviteats-thm}, there exists a matrix $M_{a,b,c}\in G(F)$ taking
$Q_0$ to $Y$. Let $\widetilde p(\bX)$ be the image of $p(\bX)$ under the action
of $M_{a,b,c}$. So we obtain $\cV\big(\widetilde p(\bX)\big)=F$. From the
above, applied to the cubic form $\widetilde p(\bX)$, we infer that the
$\widetilde p(\bX)$-tangent plane of $F$ at $Q_0$ does not pass through $Q_3$,
whence the tangent plane of $F$ at $Y$ does not pass through $Q_3=\kappa(Q_3)$
either.

\par
(IV) \emph{All points $Z\in g_\infty$ are double points of $F$.} Let $Y$ be an
affine point of $F$ and $Z\in g_\infty$. The line $YZ$ is either a generator of
$F$, or we have $YZ\cap F=\{Y,Z\}$; cf.\ formula (\ref{eq:2geradenschnitt}). If
$YZ\notin F$ then, by (III), $YZ$ meets $F$ at $Y$ and $Z$ with multiplicities
one and two, respectively. As $Y$ varies in $F\setminus g_\infty$, the lines
$YZ$ generate the whole space. Thus $Z$ cannot be a simple point.

\par

(c) The planar section $F\cap\cV(X_0)$ equals the line $g_\infty$. By (IV), all
points of $g_\infty$ are double points. Thus each line at infinity $\neq
g_\infty$ meets $F$ at a point of $g_\infty$ with multiplicity three. So
\begin{equation*}\label{}
X_1^3\sim a_{111}X_1^3+a_{112} X_1^2X_2+a_{113} X_1^2X_3+a_{122}X_1X_2^2
+a_{123}X_1X_2X_3+a_{133}X_1X_3^2,
\end{equation*}
whence
\begin{equation*}\label{}
     a_{111}\neq a_{112}=a_{113}=a_{122}=a_{123}=a_{133}=0.
\end{equation*}
Now we consider the line through $Q_3$ and a point $P(u_1,u_2)$, where
$(u_1,u_2)\in K^2$. Since $Q_3$ is a double point of $F$, the intersection
multiplicity at $P(u_1,u_2)$ equals one. This implies, for all $(u_1,u_2)\in
K^2$,
\begin{eqnarray*}\label{}
  T &\sim& p\left(1,u_1,u_2,(u_1u_2-u_1^3) + T\right)^\T \nonumber \\
    &=&
     wT   
    +a_{001} u_1
    +a_{011} u_1^{2}
    +a_{012} u_1 u_2
    +a_{111} u_1^{3}
    +w(u_1u_2-u_1^{3})
  \in K[T],
\end{eqnarray*}
where we used the shorthand $w:=a_{003}+a_{013}u_1$. Since $w$ must not vanish,
we obtain
\begin{equation*}\label{}
     a_{013}=0.
\end{equation*}
We now substitute $u_1=1$ in the constant term of
$p\big(1,u_1,u_2,(u_1u_2-u_1^3) + T\big)^\T$. Hence
\begin{equation*}\label{}
     ( a_{003}+a_{012})u_2+a_{001}-a_{003}+a_{011}+a_{111}=0 \mbox{ for all } u_2\in K,
\end{equation*}
so that
\begin{equation*}\label{}
    a_{012}=-a_{003}.
\end{equation*}
Altogether, the constant term of $p\big(1,u_1,u_2,(u_1u_2-u_1^3) + T\big)^\T$
yields the identity
\begin{equation}\label{eq:identitaet}
     ( -a_{003}+a_{111}) u_1^{3}+a_{011}u_1^{2}+a_{001}u_1=0
      \mbox{ for all } u_1\in K.
\end{equation}
There are two cases:

\par

If (\ref{eq:identitaet}) holds trivially then $a_{111}=a_{003}\neq 0$,
$a_{011}=a_{001}=0$, and $p(\bX) \sim f(\bX)$. This has to be the case when
$|K|\geq 4$.

\par

If (\ref{eq:identitaet}) is a non-trivial identity in $u_1$ then, of course,
$|K|=3$. Up to a factor $\pm 1$, $T^3+2T\in K[T]$ is the only cubic polynomial
which vanishes for all elements of $K$. So we let $a_{001}:=2$, whence
$a_{111}=1+a_{003}$. However, $a_{111}$ and $a_{003}$ must not be zero. Thus,
finally, we arrive at $a_{111}=2$ and $a_{003}=1$, as required.
\end{proof}
In the proof from the above we repeatedly used the assumption $|K|\geq 3$. If
it is dropped then the situation changes drastically.

\begin{theorem}\label{thm:K=2}
Let $|K|= 2$ and let $p(\bX)\in K[\bX]$ be a cubic form. The Cayley surface
 $F=\cV\big(f(\bX)\big)$ coincides with $\cV\big(p(\bX)\big)$ if, and only if,
\begin{equation}\label{eq:nullfkt}
f(\bX)-p(\bX)= \sum_{0\leq i<j\leq 3} b_{ij}(X_i^2X_j+X_iX_j^2)
    \mbox{ with }b_{ij}\in K=\{0,1\}.
\end{equation}
\end{theorem}
\begin{proof}
Because of $K=\{0,1\}$, $\cV\big(f(\bX)\big)=\cV\big(p(\bX)\big)$ holds
precisely when the cubic form $f(\bX)-p(\bX)\in K[\bX]$ yields the zero
function on $K^{4\times 1}$. By a result of \textsc{G.~Tallini} \cite[formula
(1)]{tallini-61a}, a cubic form in $K[\bX]$ has that property if, and only if,
it is given as in (\ref{eq:nullfkt}).
\end{proof}
By the above, we obtain $64$ cubic forms $p(\bX)$ for $|K| =2$, and we refrain
from a further discussion.

\par

If $|K|\leq 3$ then each of the polynomials $f(\bX)$ and $f_{|K|}(\bX)$ yields
the same simple (double) points and the same \emph{set\/} of tangent planes for
$F$. This is in accordance with the characterization of tangent planes in
Proposition \ref{prop:tg-ebenen}. However, for each simple point the two
polynomials yield \emph{distinct\/} tangent planes.

\par

If $|K| \geq 4$ then, by following ideas from the proof of (II), it is easy to
recover the unique point of tangency of a plane $\tau$ containing a generator
$g(1,s)$, $s\in K$, but not the point $Q_3$: We know $\tau\cap F=g(1,s)\cup k$,
where $k$ is a parabola. This $k$ is uniquely determined by $F$, because
$g(1,s)\cap\omega$ is its only point at infinity, and because there are at
least three points of $k$ outside $g(1,s)$. Thus $k$ meets $g(1,s)$ residually
at the unique point of $F$ with tangent plane $\tau$.
\end{section}

\begin{section}{Isometries of the Cayley surface $F$} \label{se:distance-function}

We shall assume $|K|\geq 4$ throughout this section.

\par

Two (possibly identical) points of $F \setminus g_{\infty}$ are said to be
\emph{parallel\/} if they are on a common generator of $F$. This parallelism is
an equivalence relation; it will be denoted by $\parallel$.


\par

Let $A=P(u_1,u_2)$ and $B=P(v_1,v_2)$, where $u_1,u_2,v_1,v_2\in K$, be
non-parallel points of $F\setminus g_\infty$. Thus $u_1\neq v_1$. The points of
intersection of the line $AB$ and $F$ are in one-one correspondence with the
zeros in $K$ of the polynomial
\begin{eqnarray*}
 \lefteqn{f\big((1-T)(1, u_1, u_2, u_1u_2-u_1^3)+T(1,v_1, v_2,
 v_1v_2-v_1^3)\big)}&& \\
 & = & T(T-1)(u_1-v_1)\big((u_1-v_1)^2T-2u_1^2+u_2+u_1v_1-v_2+v_1^2\big) \in K[T],
\end{eqnarray*}
taking into account multiplicities. We read off that those zeros are $0,1,$ and
\begin{equation}\label{distanz-P-Q-eq}
  \delta(A,B) := \frac{2u_1^2-u_2-u_1v_1+v_2-v_1^2}{(u_1-v_1)^2} \ .
\end{equation}
So $AB\cap F=\{A,B,C\}$ where, in terms of a cross ratio ($\DV$), the point $C$
is given by
\begin{equation}\label{distanz-DV-eq}
    \DV(C,B,A,I) = \delta(A,B) \ \mbox{ with } \{I\}:= AB \cap \omega.
\end{equation}
If $AB\cap F=\{A,B\}$, i.e. when $\delta(A,B)\in\{0,1\}$, our definition of $C$
is based upon the intersection multiplicity of $AB$ at $A$ and $B$. This can be
avoided as follows: By the last remark of the previous section, it is possible
to decide in a purely geometric way whether $AB$ lies in the tangent plane of
$F$ at $A$, whence $C=A$, or at $B$, whence $C=B$. (For this reason we adopted
the assumption $|K|\geq 4$.)

\par

Moreover, we define $\delta(A,B)=\infty$ whenever $A\parallel B$. So we are in
a position to regard $\delta$ as a \emph{distance function\/}
\begin{equation*}\label{}
    \delta: (F \setminus g_{\infty}) \times (F \setminus g_{\infty})
    \rightarrow K \cup \{ \infty \}.
\end{equation*}
It turns the affine part of the Cayley surface into a \emph{distance space\/}
in the sense of \textsc{W.~Benz} \cite[p.~33]{benz-92}. We obtain
\begin{equation}\label{eq:dist-formeln}
  \delta(A,A)= \infty \mbox{ and } \delta(A,B)=1- \delta(B,A)
  \mbox{ for all }A, B \in F \setminus g_{\infty},
\end{equation}
provided that we set $1- \infty:= \infty$. This distance function can be found
in a paper by \textsc{H.~Brauner} \cite[p.~115]{brau-64} for $K = \RR$ in a
slightly different form. In terms of our $\delta$, Brauner's distance function
can be expressed as
\begin{equation*}\label{}
    \widehat\delta(A,B) :=
    \frac{3}{2}\left(\frac{1}{2}-\delta(A,B)\right)^{-1};
\end{equation*}
the easy verification is left to the reader. However, the approach in
\cite{brau-64} is completely different, using differential geometry and Lie
groups. A major advantage of $\widehat\delta$ is that instead of
(\ref{eq:dist-formeln}) one obtains the much more suggestive formulas
$\widehat\delta(A,A)=0$ and $\widehat\delta(A,B)=-\widehat\delta(B,A)$. Since
we do not want to impose any restriction on the characteristic of the ground
field, it is impossible for us to make use of that function $\widehat\delta$.

\par

Given a point $A \in F \setminus g_{\infty}$ and an element $\rho \in K \cup
\{\infty\}$ we define the \emph{circle\/} with midpoint $A$ and radius $\rho$
in the obvious way as
\begin{equation*}\label{}
   \kreis(A,\rho):=\{Y \in F \setminus g_\infty \mid  \delta(A,Y)= \rho \}.
\end{equation*}
By the \emph{extended circle} $\cE(A,\rho)$ we mean the circle $\kreis(A,\rho)$
together with its midpoint $A$.

\par

If $\rho = \infty$ then $\kreis(A,\rho)=\cE(A,\rho)$ is the generator of $F$
through $A$, but without its point at infinity. In order to describe the
remaining circles, let us introduce, for $\alpha, \beta, \gamma \in K$, the
rationally parameterized curve
\begin{equation}\label{eq:alpha-beta-gamma}
 \cR_{\alpha, \beta, \gamma} :=
 \big\{ K (1, t, \alpha +\beta t+ (\gamma+1)t^2,\alpha t+ \beta t^2 + \gamma t^3)^\T
           \mid t\in K\cup\{\infty\} \big\},
\end{equation}
lying on $F$. It is a parabola for $\gamma=0$, a planar cubic for $\gamma=-1$,
and a twisted cubic parabola (i.e.\ a twisted cubic having the plane at
infinity as an osculating plane) otherwise.

\begin{lemma}\label{lem:3-punkte}
Let $P(u_{1i},u_{2i})$, $u_{ji} \in K$ with $i\in\{1,2,3\}$, be three mutually
non-parallel points of $F\setminus g_\infty$. Then there is a unique triad
$(\alpha,\beta,\gamma)\in K^3$ such that the curve $\cR_{\alpha,\beta,\gamma}$
contains the three given points.
\end{lemma}
\begin{proof}
By Lagrange's interpolation formula, there is a unique triad
$(\alpha,\beta,\gamma)\in K^3$ such that $u_{2i}=\alpha + \beta u_{1i} +
(\gamma+1)u_{1i}^2$ for $i\in\{1,2,3\}$. Hence the assertion follows.
\end{proof}
We add in passing that $F\setminus g_\infty$ together with the affine traces of
the curves (\ref{eq:alpha-beta-gamma}) is isomorphic to the affine chain
geometry on the ring $K[\eps]$ of dual numbers over $K$. An isomorphism is
given by $P(u_1,u_2)\mapsto u_1+\eps u_2$. The interested reader should compare
with \cite[p.~796]{herz-95}.

\par

Next we describe circles and extended circles:

\begin{proposition}\label{circle-subset-cubic-thm}
Suppose that a point $A=P(a_1, a_2)$, $a_1,a_2\in K$, and $\rho \in K$ are
given. Let
\begin{equation}\label{kreis=kubik-eq}
  \alpha:= (\rho-2)a_1^2+a_2,\; \beta:=(1-2\rho)a_1, \;
  \gamma:= \rho.
\end{equation}
Then $(\alpha,\beta,\gamma)$ is the only triad in $K^3$ such that the curve
$\cR_{\alpha,\beta,\gamma}$ contains the circle $\kreis(A,\rho)$. Moreover, the
extended circle $\cE(A,\rho)$ equals the set of affine points of
$\cR_{\alpha,\beta,\gamma}$.
\end{proposition}

\begin{proof}
Let $Y=P(u_1, u_2) \in F\setminus g_\infty$, where $u_1,u_2\in K$. Using
(\ref{distanz-P-Q-eq}) for $Y\notpar A$, and a direct verification otherwise,
shows that $Y\in\cE(A,\rho)$ if, and only if,
\begin{equation*}\label{}
   2a_1^2-a_2-a_1u_1+u_2-u_1^2 = \rho(a_1-u_1)^2
\end{equation*}
which in turn is equivalent to
\begin{equation*}
   u_2 = (\rho-2)a_1^2+a_2+(1-2\rho)a_1u_1 + (1+ \rho) u_1^2.
\end{equation*}
So $\cE(A,\rho)=\cR_{\alpha, \beta, \gamma}\setminus\omega$, with $\alpha,
\beta$, $\gamma$ as in (\ref{kreis=kubik-eq}). The uniqueness of $(\alpha,
\beta, \gamma)$ is immediate from $|\kreis(A,\rho)|=|\cR_{\alpha, \beta,
\gamma}|-2=|K|-1\geq 3$ and Lemma \ref{lem:3-punkte}.
\end{proof}

\begin{proposition}\label{prop:kreise}
Given a curve $\cR_{\alpha, \beta, \gamma}$, with $\alpha, \beta, \gamma \in
K$, there are three possibilities.
\begin{enumerate}
    \item $1-2\gamma \neq 0:$ $\cR_{\alpha, \beta, \gamma}\setminus\omega$
    coincides with the extended circle $\cE(A,\rho)$, where
       \begin{equation} \label{kreis=kubik-2-eq}
          A:=P\left(\frac{\beta}{1-2\gamma},\alpha-\frac{(\gamma-2)\beta^2}{(1-2\gamma)^2}\right)
                \mbox{ and }  \rho:= \gamma.
       \end{equation}

    \item $1-2\gamma = 0\neq \beta :$ $\cR_{\alpha, \beta,
    \gamma}\setminus\omega$ is not an extended circle.

    \item $1-2\gamma = 0=\beta:$ $\cR_{\alpha, \beta,
    \gamma}\setminus\omega$ is an extended circle $\cE(A,\frac{1}{2})$
    for all points $A \in \cR_{\alpha, \beta, \gamma} \setminus\omega$.

\end{enumerate}
\end{proposition}

\begin{proof}
We distinguish three cases according to the above:

\par
(a) We infer from (\ref{kreis=kubik-eq}) that $\cR_{\alpha, \beta,
\gamma}\setminus\omega = \cE(A,\rho)$, with $\rho$ and $A$ as in
(\ref{kreis=kubik-2-eq}).

\par
(b) Assume to the contrary that $\cR_{\alpha, \beta, \gamma}\setminus\omega =
\cE(A,\rho)$. Now $1-2\gamma=0$ yields $2\gamma \ne 0$, so that $\Char K \neq
2$ and $\gamma = \frac{1}{2}$. Applying Theorem \ref{circle-subset-cubic-thm}
to $\kreis(A,\rho)$ yields $\rho=\gamma=\frac{1}{2}$ and, consequently,
$\beta=0$, an absurdity.

\par
(c) We proceed as in (b) thus obtaining $\Char K \ne 2$,
$\rho=\gamma=\frac{1}{2}$, and $a_2=\alpha+\frac32 a_1^2$, where $a_1\in K$ can
be chosen arbitrarily. This means that every point $A=P(a_1,a_2)$ of the given
curve $\cR_{\alpha,\beta,\gamma}$ can be considered as midpoint.
\end{proof}
As an application of the distance function $\delta$, we investigate various
actions of the group $G(F)$ arising from its action on the projective space
$\PP_3(K)$. Given a point $P \in F \setminus g_{\infty}$ and a line $g \subset
F$ with $P \not\in g \ne g_{\infty}$ the pair $(P,g)$ will be called an
\emph{antiflag\/} of $F\setminus g_\infty$. Following \cite[p.~33]{benz-92} an
\emph{isometry\/} of $F\setminus g_\infty$ is just a mapping $\mu:F\setminus
g_\infty\to F\setminus g_\infty$ such that
$\delta(A,B)=\delta\big(\mu(A),\mu(B)\big)$ for all $A,B\in F\setminus
g_\infty$.

\begin{theorem}\label{thm:trans}
The matrix group $G(F)$ has the following properties:
\begin{enumerate}
\item\label{thm:trans.isom}

$G(F)$ acts on $F\setminus g_\infty$ as a group of isometries.

\item\label{thm:trans.antiflagge}

$G(F)$ acts regularly on the set of antiflags of $F\setminus g_\infty$.

\item \label{thm:trans.nichtpar}

For each $d\in K$ the group $G(F)$ acts regularly on the set
 \begin{equation*}\label{}
   \Delta_d:=\{(A,B)\in(F\setminus g_\infty)^2\mid  \delta(A,B)=d\}.
 \end{equation*}

\item \label{thm:trans.par}

Let $A\parallel B$ and $A'\parallel B'$ be points of $F \setminus g_{\infty}$.
Write $A=P(u_1, u_2)$, $B=P(u_1, v_2)$, $A'=P(u_1',u_2')$, and $B'=P(u_1',
v_2')$ with $u_1,u_2,\ldots,v_2'\in K$. Then the number of matrices in $G(F)$
mapping $(A,B)$ to $(A',B')$ equals the number of distinct elements $c\in
K^\times$ such that
\begin{equation}\label{kollanzahl-parallele-paare-eq}
   c^2({v_2-u_2})=({v_2'-u_2'}).
\end{equation}
\end{enumerate}
\end{theorem}

\begin{proof} (a) Let $A,B\in F\setminus g_\infty$. Suppose that $M\in G(F)$ takes $A$ to
$A'$ and $B$ to $B'$. If $\delta(A,B) \neq 0,1,\infty$ then the line $AB$ meets
the Cayley surface at three distinct points $A, B$, and $C$, say. Since $M$
preserves cross ratios, $\delta(A',B')= \delta(A,B)$ is immediate from
(\ref{distanz-DV-eq}). If $\delta=0$ then $AB$ is a tangent of $F$ at $A$. By
the remark below (\ref{distanz-DV-eq}), this tangent is mapped to a tangent of
$F$ at $A'$, whence $\delta(A',B')= 0$, as required. The case $\delta(A,B)=1$
can be treated similarly. Finally, $\delta(A,B)=\infty$ means that $A,B$ are on
a common generator, a property which is shared by their images, whence the
assertion follows.

\par

(b) Since $G(F)$ acts transitively on $F \setminus g_{\infty}$, it is
sufficient to show that the stabilizer of $Q_0$ in $G(F)$, i.e.\ $\{ M_{0,0,c}
\mid c \in K^{\times} \}$, acts regularly on $\{ g(1,c) \mid c \in
K^{\times}\}$. In fact, if we are given generators $g(1,c_1)$ and $g(1,c_2)$
with $c_1, c_2 \in K^{\times}$ then $M_{0,0,c_2c_1^{-1}}$ is the only solution.

\par
(c) Let $(A,B)$ and $(A',B')$ be elements of $\Delta_d$. By
Lemma~\ref{transitiviteats-thm}, we may assume w.l.o.g.\ that $A=A'=P(0,0)$ .
We infer from (\ref{distanz-P-Q-eq}) that a point $Y=P(y_1,y_2)$, $y_1,y_2\in
K$, satisfies $\delta(A,Y)=d$ if, and only if, $y_2 = (d+1)y_1^2$ and $y_1\in
K^\times$. So there exist elements $u_1, u_1'\in K^{\times}$ with
\begin{equation*}\label{}
    B=P(u_1,(d+1)u_1^2), \; B'=P(u_1',(d+1)u_1'^{2}).
\end{equation*}
Letting $c:= u_1'u_1^{-1}$, the matrix $M_{0,0,c}$ has the required property.
The point $A$ and the unique generator through $B$ form an antiflag; the same
holds for $A'$ and the unique generator through $B'$. So the asserted
regularity is a consequence of (\ref{thm:trans.antiflagge}).

\par

(d) The matrix $M_{-u_1,0,1}M_{0,3u_1^2-u_2,1}\in G(F)$ maps $(A,B)$ to
$\big(P(0,0),P(0,v_2-u_2)\big)$. Similarly, we can take $(A',B')$ to
$\big(P(0,0),P(0,v_2'-u_2')\big)$ by a matrix in $G(F)$. By (\ref{eq:stabil4}),
a matrix in $G(F)$ stabilizes $P(0,0)=Q_0$ if, and only if, it has the form
$M_{0,0,c}$, where $c\in K^{\times}$. As (\ref{kollanzahl-parallele-paare-eq})
is a necessary and sufficient condition for such a matrix to take
$P(0,v_2-u_2)$ to $P(0,v_2'-u_2')$, the assertion follows.
\end{proof}
The previous result shows that the action of $G(F)$ on pairs of parallel points
depends on the square classes of $K^\times$. If $K^\times$ has a single square
class (e.g.\ when $K$ is quadratically closed or when $K$ is a finite field of
even order) then all pairs of distinct parallel points are in one orbit of
$G(F)$. If $K^\times$ has precisely two square classes and if $-1$ is not a
square (e.g.\ when $K=\RR$ or when $K$ is a finite field with $|K|\equiv 3
\pmod 4)$ then all $2$-subsets of parallel points are in one orbit of $G(F)$.

\par

We are now in a position to describe all isometries of $F\setminus g_\infty$.
Recall that we do not assume an isometry to be a bijection.

\begin{theorem}
Each isometry $\mu:F \setminus g_{\infty}\to F \setminus g_{\infty}$ is induced
by a unique matrix in $G(F)$. Consequently, $\mu$ is bijective and it can be
extended in a unique way to a projective collineation of\/ $\PP_3(K)$ fixing
the Cayley surface $F$.
\end{theorem}
\begin{proof}
By Theorem \ref{thm:trans} (\ref{thm:trans.isom}) and
(\ref{thm:trans.nichtpar}), it is sufficient to verify the assertion for an
isometry $\mu$ fixing the points $P(0,0)$ and $P(1,0)$. Since $G(F)$ acts
faithfully on $F\setminus g_\infty$, the proof will then be accomplished by
showing $\mu= \id_{F \setminus g_{\infty}}$.

\par

For all $u_2 \in K$ we obtain $\delta\big(P(1,0), P(1,u_2)\big)= \infty$,
$\delta\big(P(0,0), P(1,u_2)\big)= u_2-1$, $\delta\big(P(0,0), P(0,u_2)\big)=
\infty$ and $\delta\big(P(1,0), P(0,u_2)\big)= u_2+2$. So, by the isometricity
of $\mu$, all affine points of the generators through $P(1,0)$ and $P(0,0)$
remain fixed under $\mu$.

\par

Next choose any point $P(u_1, u_2)$, where $u_1 \in K \setminus \{0,1 \}$ and
$u_2 \in K$. We determine all $(v_2,w_2)\in K^2$ subject to $\delta\big(P(0,
v_2), P(u_1, u_2)\big)=0$ and $\delta\big(P(1, w_2), P(u_1, u_2)\big)=0$. The
unique solution is $(v_2,w_2):=(u_2-u_1^2,-u_1^2-u_1+u_2+2)$. A point $P(x_1,
x_2)$, $x_1,x_2\in K$, belongs to the circle $\kreis\big(P(0, v_2),0\big)$ if,
and only if,
\begin{equation}\label{isometrisch1-eq}
x_1 \neq 0 \mbox{ and }\delta\big(P(0, v_2),P(x_1, x_2)\big) =
    \frac{-x_1^2+x_2+u_1^2-u_2}{x_1^2}=0.
\end{equation}
Similarly, $P(x_1,x_2)\in \kreis\big(P(1,w_2),0\big)$ if, and only if,
\begin{equation*}\label{isometrisch2-eq}
x_1 \neq 1 \mbox{ and }
    \delta\big(P(1, w_2),P(x_1, x_2)\big) =
    \frac{-x_1^2-x_1+x_2+u_1^2+u_1-u_2}{(x_1-1)^2}=0.
\end{equation*}
So, if $P(x_1, x_2)$ belongs to both circles it has to satisfy
\begin{equation*}
 0  = (-x_1^2+x_2+u_1^2-u_2)-(-x_1^2-x_1+x_2+u_1^2+u_1-u_2)
    =  -x_1+u_1 \ ,
\end{equation*}
whence $x_1=u_1$ and, by (\ref{isometrisch1-eq}), $x_2=u_2$. Clearly, under
$\mu$ the two circles remain fixed, so that $\mu\big(P(u_1, u_2)\big)=P(u_1,
u_2).$
\end{proof}
In the previous theorem we used that an isometry leaves invariant \emph{all}
distances. Thus we established that $\delta$ is a defining function (see
\cite[p.~23]{benz-94}) for the group of automorphic projective collineations of
$F$. It would be interesting to know if this assumption could be weakened, for
example, by requiring that only \emph{some\/} distances are being preserved.

\end{section}

\noindent
Johannes Gmainer and Hans Havlicek\\
Institut f{\"u}r Diskrete Mathematik und Geometrie\\
Technische Universit{\"a}t Wien\\
Wiedner Hauptstra{\ss}e 8--10/104\\
A-1040 Wien\\
Austria

\noindent\texttt{gmainer@geometrie.tuwien.ac.at,\\
havlicek@geometrie.tuwien.ac.at}
\end{document}